\documentclass[ECP, preprint]{ejpecp} 

\usepackage{enumerate}  
\usepackage{tikz}
\usepackage{amssymb}
\usepackage{pgfplots}
\usepackage{pgfplotstable}

\SHORTTITLE{Invariance properties of limiting point processes}

\TITLE{Invariance properties of limiting point processes and applications to clusters of extremes} 

\AUTHORS{%
  Anja~Jan\ss en\footnote{Otto-von-Guericke University Magdeburg, FMA, Germany.
    \EMAIL{anja.janssen@ovgu.de}}
  \and 
  Johan~Segers\footnote{UCLouvain, LIDAM/ISBA, Louvain-la-Neuve, Belgium. \EMAIL{johan.segers@uclouvain.be}}}


\AMSSUBJ{60G10; 60G70} 

\SUBMITTED{July 22nd, 2022} 
\ACCEPTED{} 

\VOLUME{0}
\YEAR{2020}
\PAPERNUM{0}
\DOI{10.1214/YY-TN}

\ABSTRACT{Motivated by examples from extreme value theory, but without using the theory of regularly varying time series or any assumptions about the marginal distribution, we introduce the general notion of a cluster process as a limiting point process of returns of a certain event in a time series. We explore general invariance properties of cluster processes which are implied by stationarity of the underlying time series. Of particular interest in applications are the cluster size distributions, and we derive general properties and interconnections between the size of an inspected and a typical cluster. While the extremal index commonly used in extreme value theory is often interpreted as the inverse of a ``mean cluster size'', we point out that this only holds true for the expected value of the typical cluster size, caused by an effect very similar to the inspection paradox in renewal theory.}

\usepackage{color}

\begin{document}



\section{Introduction}
\subsection{Motivation}
When one looks at a series of sequential random observations, the interest is sometimes restricted to keeping track of the times at which a particular outcome is observed in contrast to those at which it does not happen. This means that for a stochastic process $(X_t)_{t \in \mathbb{Z}}$ on a generic probability space $(\Omega, \mathcal{A}, P)$ such that $X_t:(\Omega, \mathcal{A})\to (S, \mathcal{S})$ for some measurable space $(S, \mathcal{S})$ one fixes a set $U \in \mathcal{S}$ and looks at the resulting process $(I_t)_{t \in \mathbb{Z}}$ with
 \begin{equation}\label{Eq:SimpleBin} I_t:=\mathds{1}_{\{X_t \in U\}}, \;\;\; t \in \mathbb{Z}.\end{equation}
The process $(I_t)_{t \in \mathbb{Z}}$ is then a binary time series (see, e.g., \cite{Kedem}, \cite{Keenan}), but could equally well be interpreted as a point process, i.e. a random measure of the form
  \begin{equation}\label{Eq:RandMeas} N:(\Omega, \mathcal{A}) \to \bigl(\hat{\mathcal{N}}_\mathbb{Z}^\star, \mathcal{B}(\hat{\mathcal{N}}_\mathbb{Z}^\star)\bigr), \;\;\; N(\omega)=\sum_{t \in \mathbb{Z}} \delta_{\{t\}}\mathds{1}_{\{I_t(\omega)=1\}}, \end{equation}
where $\bigl(\hat{\mathcal{N}}_\mathbb{Z}^\star, \mathcal{B}(\hat{\mathcal{N}}_\mathbb{Z}^\star)\bigr)$ denotes the space of all simple counting measures on $\mathbb{Z}$ with its Borel $\sigma$-algebra, see \cite{DaleyVJ}, Chapter 7. Note that each element $N$ of $\hat{\mathcal{N}}_\mathbb{Z}^\star$ is uniquely determined by its support, i.e. the occurrence times of the event $U$, and so $N(\omega)$ in \eqref{Eq:RandMeas} is uniquely determined by the random set
  	\begin{equation}\label{Eq:RandSupp} C(\omega)=C\bigl((I_t(\omega))_{t \in \mathbb{Z}}\bigr):=\{t \in \mathbb{Z}: I_t(\omega)=1\}.\end{equation}
  	
Special cases of $(S, \mathcal{S})$ and $U$ in \eqref{Eq:SimpleBin} are given by $(\mathbb{R}, \mathcal{B}(\mathbb{R}))$ and $U=(u,\infty)$ or $(\mathbb{R}^d, \mathcal{B}(\mathbb{R}^d))$ and $U=\{x \in \mathbb{R}^d: \|x \|>u\}$ for some $u>0$ and some norm $\| \cdot \|$, see Figure~\ref{Fig:1} for an example. 
	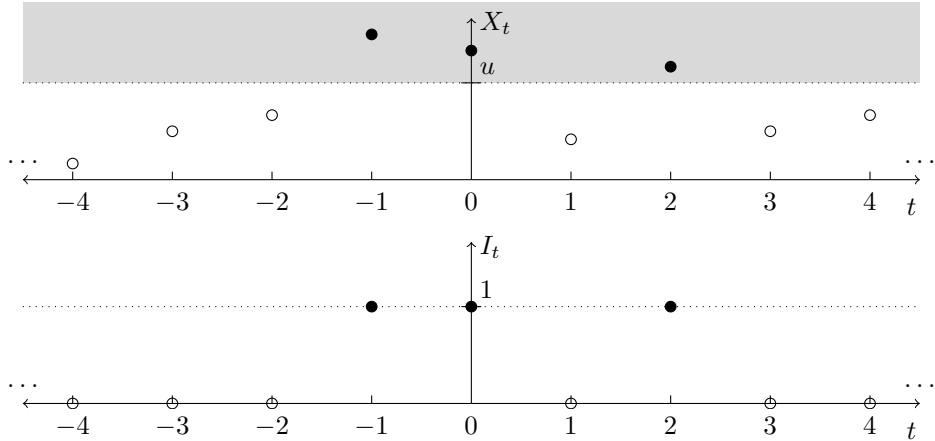
\begin{figure}
	\centering
	\begin{tikzpicture}
		\begin{axis}[
			xtick={-4,-3,-2,-1,0,1,2,3,4},
			ytick=\empty, 
			ymin=-0.3,
			ymax=11,
			xmin=-5.5,
			xmax=5.5,
			width=16cm,
			height=4cm,
			xtick style={draw=none}, 
			axis line style={draw=none}, 
			]
			
			\addplot[black, mark=*, mark options={black}, only marks] coordinates {
				
				(0, 8)
				(-1, 9)
				(2, 7)
				
			};
			\draw (axis cs:-4,1) circle[radius=2pt];
			\draw (axis cs:-3,3) circle[radius=2pt];
			\draw (axis cs:-2,4) circle[radius=2pt];
			\draw (axis cs:1,2.5) circle[radius=2pt];
			\draw (axis cs:3,3) circle[radius=2pt];
			\draw (axis cs:4,4) circle[radius=2pt];

			\draw (axis cs:1,0) -- (axis cs:1,0.5); 	
			\draw (axis cs:2,0) -- (axis cs:2,0.5); 	
			\draw (axis cs:3,0) -- (axis cs:3,0.5); 	
			\draw (axis cs:4,0) -- (axis cs:4,0.5); 				
			\draw (axis cs:-1,0) -- (axis cs:-1,0.5); 	
			\draw (axis cs:-2,0) -- (axis cs:-2,0.5); 
			\draw (axis cs:-3,0) -- (axis cs:-3,0.5); 	
			\draw (axis cs:-4,0) -- (axis cs:-4,0.5); 	
			\draw [fill=gray!30, draw=none] (axis cs:-4.5,6) rectangle (axis cs:4.5,11);			    
			\draw[->] (axis cs:0,0) -- (axis cs:4.5,0);	
			\draw[->] (axis cs:0,0) -- (axis cs:-4.5,0);	
			\draw[->] (axis cs:0,0) -- (axis cs:0,10);	
			\draw[-] (axis cs:-0.1,6) -- (axis cs:0.1,6);	
			\node[above right] at (axis cs:0,6) {$u$};	
			\node[above right] at (axis cs:0,8.5) {$X_t$};	
			\draw[dotted] (axis cs:-4.5,6) -- (axis cs:4.5,6);	
			\node[above] at (axis cs: -4.5,0.5) {$\ldots$};	
			\node[above] at (axis cs: 4.5,0.5) {$\ldots$};				
		\end{axis} 
		\node at (13,-0.3) {$t$}; 	
		
	\end{tikzpicture}
	\vspace{0.5cm}
	\begin{tikzpicture}
		\begin{axis}[
			xtick={-4,-3,-2,-1,0,1,2,3,4},
			ytick=\empty, 
			ymin=-0.3,
			ymax=11,
			xmin=-5.5,
			xmax=5.5,
			width=16cm,
			height=4cm,
			xtick style={draw=none}, 
			axis line style={draw=none}, 
			]
			
			\addplot[black, mark=*, mark options={black}, only marks] coordinates {
				
				(0, 6)
				(-1,6)
				(2, 6)
				
			};
			\draw (axis cs:-4,0) circle[radius=2pt];
			\draw (axis cs:-3,0) circle[radius=2pt];
			\draw (axis cs:-2,0) circle[radius=2pt];
			\draw (axis cs:1,0) circle[radius=2pt];
			\draw (axis cs:3,0) circle[radius=2pt];
			\draw (axis cs:4,0) circle[radius=2pt];

			\draw (axis cs:1,0) -- (axis cs:1,0.5); 	
			\draw (axis cs:2,0) -- (axis cs:2,0.5); 	
			\draw (axis cs:3,0) -- (axis cs:3,0.5); 	
			\draw (axis cs:4,0) -- (axis cs:4,0.5); 				
			\draw (axis cs:-1,0) -- (axis cs:-1,0.5); 	
			\draw (axis cs:-2,0) -- (axis cs:-2,0.5); 
			\draw (axis cs:-3,0) -- (axis cs:-3,0.5); 	
			\draw (axis cs:-4,0) -- (axis cs:-4,0.5); 	
			
			\draw[->] (axis cs:0,0) -- (axis cs:4.5,0);	
			\draw[->] (axis cs:0,0) -- (axis cs:-4.5,0);	
			\draw[->] (axis cs:0,0) -- (axis cs:0,10);	
			\draw[-] (axis cs:-0.1,6) -- (axis cs:0.1,6);	
			\node[above right] at (axis cs:0,6) {$1$};	
			\node[above right] at (axis cs:0,8.5) {$I_t$};	
			\draw[dotted] (axis cs:-4.5,6) -- (axis cs:4.5,6);	
			\node[above] at (axis cs: -4.5,0.5) {$\ldots$};	
			\node[above] at (axis cs: 4.5,0.5) {$\ldots$};				
		\end{axis} 
		\node at (13,-0.3) {$t$}; 	
		
	\end{tikzpicture}
	\caption{Above: A stretch of a realization of a time series $(X_t)_{t \in \mathbb{Z}}$ with solid dots indicating the observations in the set $U=(u,\infty)$ (shaded area) and empty dots for the remaining non-exceedances. Below: The same stretch of the corresponding realization of the time series $(I_t)_{t \in \mathbb{Z}}$. In this case, the set $C$ contains $-1,0,2$.}
	\label{Fig:1}
\end{figure}

 The resulting so-called exceedance process $(I_t)_{t \in \mathbb{Z}}$ plays an important role in extreme value theory and its applications, see \cite{FinkRoot},  \cite{HsingHueslerLeadbetter}, \cite{LeadbetterLindgrenRootzen}. There, one would typically analyze a \emph{stationary} time series $(X_t)_{t \in \mathbb{Z}}$ and look, in the univariate case, at the family of stationary exceedance processes given by
 \begin{equation}\label{Eq:ExceedSeq} I_t^n:=\mathds{1}_{\{X_t >u_n \}}, \;\;\; t \in \mathbb{Z}, \end{equation}
 for $n \in \mathbb{N}$, where $u_n$ is a sequence chosen in a way such that $P(X_0>u_n)>0$ but $P(X_0>u_n) \to 0$ as $n \to \infty$, or in other words for $u_n \uparrow F^{-1}_{X_0}(1):=\inf \{x \in \mathbb{R}: P(X_0 \leq x) =  1\}$ (with the convention $\inf \varnothing = \infty$). This approach can easily be generalized to multivariate time series $(\mathbf{X}_t)_{t \in \mathbb{Z}}$ by looking at the processes $I_t^n:=\mathds{1}_{\{\|\mathbf{X}_t\|>u_n \}}$ for some norm $\| \cdot \|$ instead, but for the remainder of this introduction we will restrict attention to the univariate case for ease of notation. 
 
 The resulting sequence of binary processes $(I_t^n)_{t \in \mathbb{Z}}$ (or, equivalently the corresponding point processes from \eqref{Eq:RandMeas} or their supports from \eqref{Eq:RandSupp}) indicate the times at which events happen which become more and more extreme and therefore increasingly unlikely. Consequentially, the finite-dimensional distributions of the processes $(I^n_t)_{t \in \mathbb{Z}}$ in \eqref{Eq:ExceedSeq} converge to those of a degenerate process consisting of only zeros. 
 
By conditioning on the event $\{I_0^n=1\}$ as $n \to \infty$, i.e. by looking at the conditional distribution
 \begin{equation}\label{Eq:GenIndSeq} \mathcal{L}\bigl((I_t^n)_{t \in \mathbb{Z}} \mid I_0^n=1 \bigr)\end{equation}
 we can avoid such a degenerate process and instead describe the limiting temporal structure that we see in rare events, given that such a rare event is observed at time $t=0$. The analysis of the resulting extremal clusters is one of the main aspects of extreme value theory for time series, starting with \cite{HsingHueslerLeadbetter}, \cite{Leadbetter} and leading to more recent results like, e.g., \cite{BasrakPlaninicSoulier}, \cite{BasrakSegers}, \cite{DombryHashorvaSoulier}, \cite{Janssen}, \cite{Planinic}. Under additional assumptions, the object of interest may in fact be not only the occurrence times of extremal events but their magnitude as well and it is possible to characterize the limiting process after a suitable, usually linear transformation, which means finding the weak limits of 
 \begin{equation*} \mathcal{L}\left(\left(\frac{X_t-a(u_n)}{b(u_n)}\right)_{t \in \mathbb{Z}}\; \bigg| \; X_0>u_n\right)\end{equation*}
 for suitably chosen $a:(-\infty, F_{X_0}^{-1}(1)) \to \mathbb{R}$ and $b:(-\infty, F_{X_0}^{-1}(1)) \to (0, \infty)$ as $u_n \uparrow F^{-1}_{X_0}(1)$. A typical form of such additional assumptions would be the framework of stationary regularly varying time series $(X_t)_{t \in \mathbb{Z}}$ (see \cite{KulikSoulier} for a general introduction), where we restrict our attention here to a univariate, non-negative time series $(X_t)_{t \in \mathbb{Z}}$. We call such a time series regularly varying if the marginal distribution of $X_0$ has a univariate regularly varying tail with some index $\alpha>0$, i.e.
 $$ \lim_{x \to \infty}\frac{P(X_0>\lambda x)}{P(X_0>x)}= \lambda^{-\alpha}, \;\;\; \mbox{ for all }\lambda>0,$$
 and there exists a non-degenerate, so-called tail process $(Y_t)_{t \in \mathbb{Z}}$, see \cite{BasrakSegers}, such that
 \begin{equation}\label{Eq:TailProc} \mathcal{L}\left(\left(\frac{X_t}{u}\right)_{t \in \{-m, \ldots, n\}}\; \bigg| \; X_0>u\right)  \overset{w}{\Longrightarrow} \mathcal{L}\left((Y_t)_{t \in \{-m, \ldots, n\}}\right)\end{equation}
 for all $m,n \in \mathbb{N}$ as $u \to \infty$, where $\overset{w}{\Longrightarrow}$ stands for convergence in distribution. Key properties of the tail process are that 
 \begin{equation}\label{Eq:TailProp} P(Y_0>x)=1-x^{-\alpha},\;\; x \geq 1, \mbox{ and } Y_0 \mbox{ is independent of } (Y_t/Y_0)_{t \in \mathbb{Z}}. \end{equation}  
 
 Furthermore, the stationarity of $(X_t)_{t \in \mathbb{Z}}$ influences the structure of the tail process $(Y_t)_{t \in \mathbb{Z}}$, leading to the so-called ``time change formula''
  \begin{equation}\label{Eq:TCF}
  	E\left(f\left((s Y_{t-i})_{t \in \mathbb{Z}}\right)\mathds{1}_{\{s Y_{-i}>1\}}\right)=s^\alpha E\left(f\left((Y_{t})_{t \in \mathbb{Z}}\right)\mathds{1}_{\{Y_i>s\}}\right),
  	\end{equation}
  which holds for all $s>0$ and all bounded, measurable functions $f: \mathbb{R}^{\mathbb{Z}} \to \mathbb{R}$, see \cite[Theorem 3.1]{BasrakSegers} and also \cite[Theorem 5.3.1]{KulikSoulier}. 
  Thus, for a non-negative, stationary regularly varying time series $(X_t)_{t \in \mathbb{Z}}$ with tail process $(Y_t)_{t \in \mathbb{Z}}$, setting $I_t^n=\mathds{1}_{\{X_t>u_n\}}$ for a sequence $u_n \to \infty$  implies the existence of a limiting process $(I_t)_{t \in \mathbb{Z}}:=(\mathds{1}_{\{Y_t>1\}})_{t \in \mathbb{Z}}$ such that 
  $$ \mathcal{L}\left((I_t^n)_{t \in \{-u, \ldots, v\}} \mid I_0^n=1 \right) \overset{w}{\Longrightarrow} \mathcal{L}\left((I_t)_{t \in \{-u, \ldots, v\}}\right), $$
  as \eqref{Eq:TailProp} implies that 1 is a continuity point for each $\mathcal{L}(Y_t), t \in \mathbb{Z}$. Furthermore, for any $A \subset \mathbb{Z}$ with $0 \in A$, choosing $s=1$,
$$ f((y_t)_{t \in \mathbb{Z}}):=\begin{cases}1, & \mbox{ if } y_t>1 \mbox{ for all } t\in A \\
  0, & \mbox{ else }
 \end{cases} $$
 and $i \in A$ in \eqref{Eq:TCF} gives 
 \begin{equation}\label{Eq:InvarI} P(I_{t-i}=1 \text { for all }  t \in A)=P(I_t=1 \text { for all }  t \in A). \end{equation}
 We will see in the following that an equality of the form \eqref{Eq:InvarI} holds in fact under much more general assumptions. 
  
 \subsection{Assumptions and Overview}
 The aim of this note is to take a step back from typical frameworks in extreme value theory which impose certain assumptions about the marginal distribution of a time series and its standardized exceedances and find similar but much more general invariance properties of what we will call cluster processes. While it might be helpful to keep in mind the results from extreme value theory mentioned in the previous paragraph and in particular the exceedance processes of the form \eqref{Eq:ExceedSeq} as a typical example of such a cluster process, we will derive invariance properties of the limiting distribution of \eqref{Eq:GenIndSeq} \emph{under minimal requirements}, making use only of the stationarity of the underlying binary processes $(I_t^n)_{n \in \mathbb{N}}$ and their conditional convergence, as formalized in the following assumption.
 We write $\mathbb{N} = \{1, 2, \ldots\}$ and $\mathbb{N}_0 = \mathbb{N} \cup \{0\}$.
 
 \begin{assumption}\label{Ass:1} Let $(I_t)_{t \in \mathbb{Z}}$ be a stochastic process such that there exists a family $(I_t^n)_{t \in \mathbb{Z}}, n \in \mathbb{N},$ of $\{0,1\}$-valued stationary stochastic processes such that $P(I_0^n=1)>0$ for all $n \in \mathbb{N}$ and
 	 $$ \mathcal{L}\left((I_t^n)_{t \in \{-u, \ldots, v\}} \mid I_0^n=1\right) \overset{w}{\Longrightarrow} \mathcal{L}\left((I_t)_{t \in \{-u, \ldots, v\}}\right), \;\;\; n \to \infty,$$
 	 for all $u,v \in \mathbb{N}$. 
 	\end{assumption}
 
 Observe that Assumption~\ref{Ass:1} immediately implies that $I_0=1$ almost surely. It should also be noted that the process $(I_t)_{t \in \mathbb{Z}}$ is in general not stationary. While Assumption~\ref{Ass:1} allows for $\mathcal{L}((I_t^n)_{t \in \mathbb{Z}}) \overset{w}{\Longrightarrow} \mathcal{L}((J_t)_{t \in \mathbb{Z}})$ for some stationary process $(J_t)_{t \in \mathbb{Z}}$ with $P(J_0=1)>0$, the setting in which $P(I_0^n =1) \to 0$ would typically lead to more interesting results as otherwise Theorem~\ref{Th:TCF} below is a rather trivial statement and the cluster sizes treated in Sections~\ref{Sect:ICSize} and~\ref{Sect:TCSize} are almost surely infinite.

 The process $(I_t)_{t \in \mathbb{Z}}$ can be interpreted as indicating the times at which (in the limit) an event returns, given that we observe it at time $t=0$. If $P(I_0^n=1) \to 0$ for $n \to \infty$, then an underlying dependence structure in each process $(I_t^n)_{t \in \mathbb{Z}}$ can, via conditioning on $\{I_0^n=1\}$, prevent the process $(I_t)_{t \in \mathbb{Z}}$ from being equal to $(\mathds{1}_{\{t=0\}})_{t \in \mathbb{Z}}$. Typically, it would be an underlying factor causing the process $(I_t)_{t \in \mathbb{Z}}$ to take the value 1 exactly at time $t=0$ but possibly also within a certain time span before and after $t=0$. However, if we assume that the impact of such an effect fades out as time passes, it is often reasonable to assume that 
 \begin{equation} \label{Eq:FinClus} 
 	P \Bigl(\lim_{t \to \infty}I_t=0 \Bigr)
 	= P \Bigl(\lim_{t \to -\infty}I_t=0 \Bigr)=1, 
 \end{equation}
 i.e. that each such episode is almost surely of finite length. Even though we will not assume a priori that \eqref{Eq:FinClus} holds, it motivates us to speak of a \emph{cluster} of observations, for which we only keep track of the times at which we see an observation belonging to this cluster. Therefore, we will name the process $(I_t)_{t \in \mathbb{Z}}$ the cluster process and the random set $C(\omega)=\{t \in \mathbb{Z}: I_t(\omega)=1\}$ the cluster.
 
 In the following section, we will derive invariance principles for the process $(I_t)_{t \in \mathbb{Z}}$. The results are related to the phenomenon expressed in the time change formula~\eqref{Eq:TCF} and in~\eqref{Eq:InvarI} but hold under Assumption~\ref{Ass:1} only. Next, we will use these principles to characterize the distribution of cluster sizes. To this end, we will introduce the two quantities of the size of an \emph{inspected} cluster and the size of a \emph{typical} cluster (see also \cite{Planinic}) in Sections~\ref{Sect:ICSize} and~\ref{Sect:TCSize}, respectively, and explore how they are related. Next, we interpret the results in the framework of extreme value theory, explain the relationship to the notion of the extremal index and derive an inequality between expected values of inspected and typical cluster sizes which is similar to the inspection paradox in renewal theory. In the final section, two simple examples illustrate our results.
 
\section{An invariance principle for the cluster process}

The following invariance principle for $(I_t)_{t \in \mathbb{Z}}$ follows from the stationarity of the underlying processes $(I_t^n)_{t \in \mathbb{Z}}$. 

 \begin{theorem}[Time change formula for $(I_t)_{t \in \mathbb{Z}}$]\label{Th:TCF}
Let $(I_t)_{t \in \mathbb{Z}}$ be as in Assumption \ref{Ass:1} and let $A \subset \mathbb{Z}$ with $0 \in A$. Then,
 \[
	P(I_t=1 \text { for all }  t \in A)=P(I_{t-a}=1 \text { for all }  t \in A) \;\;\; \mbox{ for all } a \in A.
 \]
 \end{theorem}

 \begin{proof}
 Assume first that $|A|<\infty$, where $|A|$ denotes the cardinality of the set $A$. Let $a \in A$. Then, $P(I_t=1 \text { for all }  t \in A)$ and $P(I_{t-a}=1 \text { for all }  t \in A)$  are determined by the finite-dimensional distributions of the random vectors $(I_t)_{t \in A}$ and $(I_{t-a})_{t \in A}$, respectively. We thus get from Assumption~\ref{Ass:1} and from $0, a \in A$ that
 \begin{eqnarray*}P(I_t=1 \text { for all }  t \in A)
 	&=& \lim_{n \to \infty} \frac{P(I_t^n=1 \text { for all }  t \in A)}{P(I_0^n=1)}\\
 	&=&\lim_{n \to \infty} \frac{P(I_t^n=1 \text { for all }  t \in A, I_0^n=1, I_a^n=1)}{P(I_0^n=1)}.
 	\end{eqnarray*}
By stationarity of all $(I_t^n)_{t \in \mathbb{Z}}$ and since $0, a \in A$, the right-hand side equals
\begin{align*} 
	\lim_{n \to \infty} \frac{P(I_{t-a}^n=1 \text { for all }  t \in A, I_{-a}^n=1, I_0^n=1 )}{P(I_0^n=1)}
	&=\lim_{n \to \infty} \frac{P(I_{t-a}^n=1 \text { for all }  t \in A, I_0^n=1 )}{P(I_0^n=1)}\\ &=P(I_{t-a}=1 \text { for all }  t \in A),
\end{align*} 
which proves the statement in case $|A|<\infty$.
For general $A \subset \mathbb{Z}$ and $a \in A$, set $A_n:=A \cap [-\max(n,|a|),  \max(n,|a|)]$ for which $|A_n|<\infty$, $A_n \uparrow A$ as $n \to \infty$, and $a \in A_n$ for all $n \in \mathbb{N}$, so that
\begin{eqnarray*} P(I_t=1 \text { for all }  t \in A) &=& \lim_{n \to \infty} P(I_t=1 \text { for all }  t \in A_n)
\\
&=& \lim_{n \to \infty} P(I_{t-a}=1 \text { for all }  t \in A_n) = P(I_{t-a}=1 \text { for all }  t \in A)\end{eqnarray*}
from the previous case and by continuity from above. 
 \end{proof}

In terms of the random set $C$ in \eqref{Eq:RandSupp}, the conclusion of Theorem~\ref{Th:TCF} can be written as $P(A \subset C) = P(A - a \subset C)$ for $A \subset \mathbb{Z}$ such that $0, a \in A$. This equality foreshadows the following corollary, which is helpful for proving more refined statements about the times $t$ at which the event $\{I_t=1\}$ is observed, as will be needed in the next section. 

\begin{corollary}\label{Cor:TCFSupp} Let $(I_t)_{t \in \mathbb{Z}}$ be as in Assumption \ref{Ass:1} with $C(\omega)=\{t \in \mathbb{Z}: I_t(\omega)=1\}$.
	\begin{enumerate}[(a)]
		\item Let $A \subset \mathbb{Z}$ with $0 \in A$. Then
	$$ P( C = A)=P(C =  A-a) \;\;\; \mbox{ for all } a \in A, $$	
	where $A-a:=\{t \in \mathbb{Z}:a+t \in A\}$. 
	    \item Let $A \subset \mathbb{Z}$ with $0 \in A$ and let $t_1, t_2 \in \mathbb{Z}$ with $t_1 \leq 0 \le t_2$. Then,
	    \[ 
	    	P \bigl( C \cap [t_1,t_2] = A \cap [t_1,t_2] \bigr)
	    	= P \bigl( C \cap [t_1-a, t_2-a] =  (A-a) \cap [t_1-a,t_2-a] \bigr)
	    \]
	     for all $a \in A \cap [t_1, t_2]$.
	\end{enumerate}
	\end{corollary}

\begin{proof}
We start by showing (b). Let $t_1, t_2 \in \mathbb{Z}$ be such that $t_1 \le 0 \le t_2$ and let $a \in A \cap [t_1, t_2]$. By the inclusion--exclusion principle and by applying Theorem~\ref{Th:TCF} in the third step below, we find
\begin{align*}
 \lefteqn{P \bigl( C \cap [t_1,t_2] = A \cap [t_1,t_2] \bigr)}\\
 &= P(I_t =1 \mbox{ for all } t \in  A \cap [t_1,t_2], I_t=0 \mbox{ for all } t \in  \{t_1, \ldots, t_2\} \setminus A),\\
  &= \sum_{K \subset \{t_1, \ldots, t_2\} \setminus A}(-1)^{|K|} P \bigl( I_t=1 \text { for all }  t \in (A \cap [t_1,t_2])\cup K \bigr) \\
  &= \sum_{K \subset \{t_1, \ldots, t_2\} \setminus A}(-1)^{|K|} P \bigl( I_{t-a}=1 \text { for all }  t \in (A \cap [t_1,t_2])\cup K \bigr) \\
 &= P \bigl(I_t =1 \mbox{ for all } t \in  (A-a) \cap [t_1-a,t_2-a], \\
 & \hspace{1cm}I_t=0 \mbox{ for all } t \in  \{t_1-a, \ldots, t_2-a\} \setminus (A-a)\bigr),\\
 &= P \bigl(C \cap [t_1-a, t_2-a]= (A-a) \cap [t_1-a, t_2-a] \bigr).
 \end{align*}
Part~(a) follows from part~(b) by letting $t_1 \to -\infty, t_2 \to \infty$ and using continuity from above. 
	\end{proof}

\begin{remark} Invariance properties similar to Theorem~\ref{Th:TCF} are frequent topics in the theory of point processes and appear in different forms. One can show that Theorem~\ref{Th:TCF} implies that the process $(I_t)_{t \in \mathbb{Z}}$ is \emph{point-stationary} in the sense of \cite{Thorisson}, Chapter~9, and \cite{LastThorisson}. However, note that our Theorem~\ref{Th:TCF} derives its invariance property for a \emph{limiting} process derived from a sequence of stationary point processes, in contrast to the aforementioned connections between different forms of stationarity for \emph{one underlying} process. 
It is furthermore easy to show that, even if we have not used an assumption about the regular variation of the underlying time series $(X_t)_{t \in \mathbb{Z}}$ in the first place, the process $(Y_t):=(YI_t)_{t \in \mathbb{Z}}$ for some $Y$ with $P(Y>x)=1-x^{-\alpha}, x \geq 1,$ independent of $(I_t)_{t \in \mathbb{Z}}$ is a valid tail process of a regularly varying time series, see \cite{Janssen}. This allows to apply in particular the results of \cite{Planinic} to our setting. The concept of ``exceedance stationarity'' introduced there describes the same invariance property that we see in $(I_t)_{t \in \mathbb{Z}}$ if one restricts attention to the special case of a tail process of the form $(YI_t)_{t \in \mathbb{Z}}$.
\end{remark}
In the following sections we derive some further properties of $(I_t)_{t \in \mathbb{Z}}$ from Theorem~\ref{Th:TCF}. 

\section{The inspected cluster size distribution}\label{Sect:ICSize}

Without further knowledge of the processes $(I_t^n)_{t \in \mathbb{Z}}$ it is impossible to give more details on the particular form of the cluster process $(I_t)_{t \in \mathbb{Z}}$ beyond the invariance principle from the previous section. Yet, as we will see in this section, it is possible to derive more general statements about summary statistics derived from $(I_t)_{t \in \mathbb{Z}}$. Based on the interpretation of $(I_t)_{t \in \mathbb{Z}}$ as a cluster of events, the size of such a cluster would typically be of interest. Here, we should pay attention to the fact that $(I_t)_{t \in \mathbb{Z}}$ describes a cluster from which we know that it includes the arbitrarily chosen but fixed view point $t=0$, and so it makes sense to speak more precisely of an \emph{inspected cluster}. To measure the size of an inspected cluster we introduce the quantity 
\begin{equation}\label{Eq:clustersize} S^i(\omega):=|C(\omega)|= \sum_{t \in \mathbb{Z}}I_t \in \mathbb{N} \cup \{\infty\}, \end{equation}
which we call the \emph{inspected cluster size}. Note that for the exceedance process as introduced in \eqref{Eq:ExceedSeq}, several rather mild dependence conditions on the process $(X_t)_{t \in \mathbb{Z}}$ lead to $P(S^i<\infty)=1$ (see \cite{HsingHueslerLeadbetter}), but we do not rule out the case $P(S^i=\infty)>0$ a priori. 

Furthermore, given the reference point $t=0$, it may be of interest to keep track of how many instances of a cluster happen before and after this reference point, so we introduce
\begin{equation}
\label{Eq:FWBWclustersize} 
S_+^i(\omega):=\sum_{t=1}^\infty I_t \in \mathbb{N}_0 \cup \{\infty\}, \;\;\; 
S_-^i(\omega):=\sum_{t=1}^\infty I_{-t} \in \mathbb{N}_0 \cup \{\infty\}, \end{equation}
noting that
 \begin{equation}\label{Eq:Splusminus} S^i(\omega)=S_-^i(\omega) + 1 + S_+^i(\omega). 
 \end{equation}

We start by showing that the joint distribution of $(S_+^i,S_-^i)$ (and thereby also the distribution of $S^i$) is already determined by the distribution of $S_+^i$ or $S_-^i$, respectively. 

\begin{proposition}\label{Prop:BeforeAfter}
Let $(I_t)_{t \in \mathbb{Z}}$ be as in Assumption~\ref{Ass:1} and $S^i, S_+^i, S_-^i$ as defined in \eqref{Eq:clustersize}--\eqref{Eq:FWBWclustersize}. Then
\begin{equation}\label{eq:BeforeAfter} P(S_+^i \geq k, S_-^i \geq \ell)= P(S_+^i \geq k+\ell)= P(S_-^i \geq k+\ell), \;\;\; k,\ell \in \mathbb{N}_0. \end{equation}
\end{proposition} 

\begin{proof}
	Let 
	$$T_i(\omega):=\begin{cases}
	\inf\{n \in \mathbb{N}_0: |C(\omega) \cap [0,n]| \geq i+1 \}, & \mbox{ for } i \geq 0 \\
	-\inf\{n \in \mathbb{N}_0: |C(\omega) \cap [-n,0]| \geq |i|+1 \}, & \mbox{ for } i < 0 \\
	\end{cases}$$
	with $T_i(\omega) \in \mathbb{Z} \cup \{-\infty, \infty\}$ denote the time at which the $|i|$-th 1 is observed in $(I_t)_{t \in \mathbb{Z}}$ before (if $i<0$) or after (if $i> 0$) time $t=0$, where $\inf \emptyset = +\infty$ by convention; note that $T_0 = 0$, since $0 \in C$. Using Corollary~\ref{Cor:TCFSupp}(b) in the third step below we can thus write
    \begin{eqnarray*} \lefteqn{P(S^i_+ \geq k, S^i_-\geq \ell)} \\
    	&=&\sum_{t_\ell^-<\ldots<t_0^-=0=t_0^+<\ldots<t_k^+} P(T_i=t_i^+, i=0, \ldots, k, \mbox{ and } T_{-j}=t_j^-, j=0, \ldots, \ell)\\
    &=& \sum_{t_\ell^-<\ldots<t_0^-=0=t_0^+<\ldots<t_k^+}P(C\cap [t_\ell^-,t_k^+]=\{t_\ell^-, \ldots, 0, \ldots, t_k^+\}) \\    
&=& \sum_{t_\ell^-<\ldots<t_0^-=0=t_0^+<\ldots<t_k^+} P(C\cap [0,t_k^+ -t_\ell^-]=\{0, \ldots, -t_\ell^-, \ldots, t_k^+-t_\ell^-\}) \\
&=& \sum_{0=s_0^+<\ldots<s_k^+<\ldots <s_{k+\ell}^+} P(C\cap [0,s_{k+\ell}^+]=\{0, \ldots, s_k^+, \ldots, s_{k+\ell}^+\}) \\
&=& \sum_{0=s_0^+<\ldots<s_k^+<\ldots <s_{k+\ell}^+} P(T_i=s_{i}^+, i=0, \ldots, k + \ell) = P(S^i_+ \geq k+\ell).
\end{eqnarray*}
Showing that $P(S_+^i \geq k, S_-^i \geq \ell)= P(S_-^i \geq k+\ell)$ can be done analogously as above or simply by noting that if $(I_t)_{t \in \mathbb{Z}}$ satisfies Assumption~\ref{Ass:1}, then so does $(I_{-t})_{t \in \mathbb{Z}}$. 
	\end{proof}

Proposition \ref{Prop:BeforeAfter} has a few direct implications for the joint distribution of $(S_+^i, S_-^i)$. 
\begin{theorem}
	\label{Th:timechange:eq}
	Let the assumptions of Proposition~\ref{Prop:BeforeAfter} hold.  
	\begin{enumerate}[(a)]
	\item The law of $(S^i_-, S^i_+)$ is symmetric and for all $k, \ell \in \mathbb{N}_0$, we have
	\begin{align}
	\label{eq:Npm:eqge}
	P( S^i_+ = k, S^i_- \ge \ell)
	&= P( S^i_{\pm} = k + \ell), \\
	\label{eq:Npm:eq}
	P( S^i_+ = k, S^i_- = \ell ) 
	&= P( S^i_{\pm} = k + \ell ) - P( S^i_{\pm} = k + \ell + 1),
	\end{align}
	where the variable $S^i_{\pm}$ on the right-hand side can represent both $S^i_-$ and $S^i_+$.
	
	\item The common probability mass function of $S^i_-$ and $S^i_+$ is nonincreasing on~$\mathbb{N}_0$:
	\begin{equation}
	\label{eq:N:noninc}
	P(S^i_{\pm} = k) \ge P(S^i_{\pm} = k+1), \qquad k \in \mathbb{N}_0.
	\end{equation}
	\end{enumerate}
\end{theorem}

\begin{proof}
	(a) Equation~\eqref{eq:BeforeAfter} readily implies that $P(S^i_+ \ge k, S^i_- \ge \ell ) =  P(S^i_+ \ge \ell, S^i_- \ge k)$ for all $k, \ell \in \mathbb{N}_0$. As a consequence, the law of $(S^i_-, S^i_+)$ is symmetric and in particular $S^i_+$ and $S^i_-$ have the same distribution, so the distribution of the random variable $S^i_{\pm}$ is well defined. Now, for all $k, \ell \in \mathbb{N}_0$, we have, by \eqref{eq:BeforeAfter}, replacing `$S^i_+$' by `$S^i_{\pm}$' on the right-hand side,
	\begin{align*}
		P( S^i_+ = k, S^i_- \ge \ell )
		&= P( S^i_+ \ge k, S^i_- \ge \ell ) - P( S^i_+ \ge k+1, S^i_- \ge \ell ) \\
		&= P( S^i_{\pm} \ge k + \ell) - P( S^i_{\pm} \ge k + \ell + 1 ) = P( S^i_{\pm} = k + \ell ),
	\end{align*}
     which proves \eqref{eq:Npm:eqge} and further implies
     \begin{align*}
     P( S^i_+ = k, S^i_- = \ell)
     &= P( S^i_+ = k, S^i_- \ge \ell ) - P( S^i_+ = k, S^i_- \ge \ell + 1 ) \\
     &= P( S^i_{\pm} = k + \ell ) - P( S^i_{\pm} = k + \ell + 1 ),
     \end{align*}
     which proves \eqref{eq:Npm:eq}. 
    
    (b) Setting $\ell =0$ in \eqref{eq:Npm:eq} yields
    $$ 0 \leq 	P( S^i_+ = k, S^i_-=0) = P( S^i_{\pm} = k ) - P( S^i_{\pm} = k + 1) $$
    for all $k \in \mathbb{N}_0$ and thus the monotonicity of the probability mass function of $S^i_{\pm}$.
\end{proof}

The final result of this section allows us to conveniently derive the distribution of $S_{\pm}^i$ from that of $S^i$ and vice versa.

\begin{theorem}\label{Th:FinClusUnif} 	Let the assumptions of Proposition \ref{Prop:BeforeAfter} hold.   
	\begin{enumerate}[(a)]
	\item For all $k \in \mathbb{N}$, we have 
	\begin{equation} P(S^i = k) = k \cdot [P(S_\pm^i = k-1) - P(S_\pm^i = k)].\end{equation} 
	\item For all $k \in \mathbb{N}$ with $P(S^i=k)>0$, we have
	\begin{equation}\label{Eq:Spmcond} P(S^i_{\pm}=\ell \mid S^i=k)=\frac{1}{k}, \;\;\; \ell=0, \ldots, k-1. \end{equation}
	\item If $P(S^i=\infty)>0$, then 
	\begin{equation} P(S^i_{\pm}=\infty \mid S^i=\infty)=1. \end{equation}
	\end{enumerate}
\end{theorem}

\begin{proof}
	(a) Let $k \in \mathbb{N}$ and use identity~\eqref{eq:Npm:eq} to find \begin{align*}
			P(S^i=k)&=\sum_{j=0}^{k-1}P(S^i_+=j, S^i_-=k-j-1) = \sum_{j=0}^{k-1}[P(S^i_{\pm}=k-1)-P(S^i_{\pm}=k)] \\
			&= k \cdot [P(S^i_{\pm}=k-1)-P(S^i_{\pm}=k)].
			\end{align*}		
			
	(b) Let $k \in \mathbb{N}$ with $P(S^i=k)>0$ and $\ell \in \{0, \ldots, k-1\}$. Then, by making use of identity~\eqref{eq:Npm:eq} and (a), we have
	\begin{align*}
		P(S^i_{\pm}=\ell \mid S^i=k)&= \frac{P(S^i_{\pm}=\ell, S^i=k)}{P(S^i=k)}= \frac{P(S^i_+=\ell, S^i_-=k-\ell-1)}{P(S^i=k)} \\
		&= \frac{P(S_\pm^i = k-1) - P(S_\pm^i = k)}{k \cdot [P(S_\pm^i = k-1) - P(S_\pm^i = k)]}=\frac{1}{k}.
	\end{align*}
	
	(c) Assume that $P(S^i=\infty)>0$. Then for all $\ell \in \mathbb{N}_0$, we have
	\begin{align*}
	P(S^i_{\pm}=\ell, S^i=\infty)&= \lim_{k \to \infty}P(S^i_+=\ell, S^i\geq k) = \lim_{k \to \infty}P(S^i_+=\ell, S^i_-\geq k-\ell-1) \\
	&= \lim_{k \to \infty}P(S^i_{\pm}=k-1),
		\end{align*}
	where in the last step \eqref{eq:Npm:eqge} was used. Summability of the probability mass function implies that $\lim_{k \to \infty}P(S^i_{\pm}=k-1) =0$ and thus $P(S^i_{\pm}=\ell, S^i=\infty)=0$ for all $\ell \in \mathbb{N}_0$, which leads to the desired result.
	\end{proof}

Taken together, we see that the size of an inspected cluster and the number of observations before and after the \emph{inspection point} $t = 0$ can be modelled as a two-step procedure: First, the inspected cluster size is determined according to the distribution of $S^i$. Second, the number of observations before and after the inspection point are either both infinite or determined according to a uniformly distributed random split of the inspected cluster size.  

\section{The typical cluster size distribution and the extremal index}\label{Sect:TCSize}

Given our definition of the cluster process in Assumption~\ref{Ass:1}, it is natural to proceed as in the previous section and analyze the inspected cluster size. It should, however, be noted that the resulting quantities and their distributions may not correspond to the understanding of a cluster as desired in applications, as the conditioning on $\{I_0^n=1\}$ in Assumption~\ref{Ass:1} induces a size-bias similar to the inspection paradox, see \cite[Chapter I.4]{Feller}, and \cite{Planinic} for the size-bias observed in tail processes of regularly varying time series. To illustrate this point, think again of the exceedance processes of \eqref{Eq:ExceedSeq} and assume that $X_t$ stands for water level at a site protected by a dike. If the water level is extremely high for several days in a row, the dike will be weakened by pressure and moisture and may break. Thus, to construct a safe dike it would be important to know how long, on average, an extremal episode of high water levels lasts. From knowledge of the distribution of the inspected cluster size $S^i$ one could now be tempted to use 
  \begin{equation}\label{Eq:EInspC}  E(S^i)=E(S^i_-)+1+E(S^i_+)=1+2 E(S^i_{\pm}) \end{equation}
  as the relevant quantity. It has to be noted, however, that this is the expected value of the overall length of an \emph{inspected} cluster, i.e.\ one that we observe at a random time during an extremal episode and thus typically already after some extremal events have happened previous to the inspection. This random choice of inspection time induces a bias towards longer clusters, similar to the inspection or waiting-time paradox, where the size of an inspected epoch of a renewal process stochastically dominates the size of a typical epoch, see \cite{Angus}. 
  
  To avoid this bias and instead make sure that we look at a typical epoch inspected at a fixed point relative to the cluster, the most straightforward way to do so is to make sure that we inspect the extremal cluster at the beginning, i.e. at a time at which $S^i_-=0$. The idea behind this bias-correction is similar to the concept of anchoring maps, as introduced in \cite{BasrakPlaninic}. It leads us to introducing the \emph{typical} cluster size distribution given by
     \begin{equation}\label{Eq:TypClus} \mathcal{L}(S^t):=\mathcal{L}(S^i \mid S^i_-=0), \end{equation}
  if $P(S^i_-=0)>0$ (otherwise it is not defined), see also \cite{Planinic} for more on the notion of a typical cluster in the context of regularly varying time series. It is worth mentioning that Theorem~\ref{Th:FinClusUnif} implies that the distribution of $S^t$ is defined if and only if $P(S^i<\infty)>0$ and that in this case $P(S^t<\infty)=1$. Furthermore, by Theorem~\ref{Th:timechange:eq}(a), we could have equally well conditioned on $\{S^i_+=0\}$ in \eqref{Eq:TypClus}, i.e.\ on making our observation at the end of a cluster instead of at the beginning, leading to the same distribution of the typical cluster.

Again for the special case of the exceedance processes in \eqref{Eq:ExceedSeq} and under mild assumptions on $(X_t)_{t \in \mathbb{Z}}$, the probability of the conditioning event equals the well-known \emph{extremal index}, see \cite[Chapter 3.7]{LeadbetterLindgrenRootzen}, \cite{OBrien} and also \cite{Buriticaetal} for a recent overview. Therefore, we denote it by
\begin{equation}\label{Eq:ExtrIndex}
\theta := P(S^i_{\pm} = 0) = P(I_t = 0 \text { for all } t>0) = P(I_t = 0 \text { for all } t<0).
\end{equation}

A common interpretation of the extremal index is as the reciprocal of the ``expected cluster size'', see \cite{Leadbetter}. Some care has to be taken how to define this ``cluster size'', as this is indeed not $E(S^i)$ from \eqref{Eq:EInspC} but $E(S^t)$. In order to derive the latter, it is convenient to introduce, provided $\theta>0$, the quantity
 \begin{equation}\label{Eq:ClusterPMF}
 \pi_k := P(S^t = k) = P(S^i=k \mid S^i_-=0), \qquad k \in \mathbb{N},
 \end{equation}
i.e.\ the probability mass function of the typical cluster size. 
\begin{proposition} \label{Prop:ClusterPMF}
Let $(I_t)_{t \in \mathbb{Z}}$ be as in Assumption \ref{Ass:1} and assume $\theta>0$. Then, 
	\begin{equation}\label{Eq:StfromSi}
	\pi_k = \theta^{-1}[ P(S^i_{\pm} = k-1) - P(S^i_{\pm}= k) ], \qquad k \in \mathbb{N}.
	\end{equation}
Conversely, 
\begin{equation}\label{Eq:SifromSt}
	P(S^i_{\pm}=\ell) = \theta P(S^t \geq \ell +1 ) = \theta \sum_{i = \ell +1}^\infty \pi_i, \qquad \ell \in \mathbb{N}_0.
	\end{equation}
	\end{proposition}
\begin{proof}
	For $k \in \mathbb{N}$ we have
	\begin{align*}
		\pi_k
		&= \frac{P(S^i=k, S^i_-=0)}{P(S^i_-=0)} = \theta^{-1} P(S^i_+=k-1, S^i_-=0) = \theta^{-1} [P(S^i_{\pm} = k-1) - P(S^i_{\pm}= k)]
		\end{align*}
	by the definition of $\theta$ in \eqref{Eq:ExtrIndex} and by Equation~\eqref{eq:Npm:eq}, which gives \eqref{Eq:StfromSi}. Furthermore, since $P(S^t<\infty)=1$ and $P(S^i_+=\infty,S^i_-=0)=0$ by Theorem \ref{Th:FinClusUnif}, we have for $\ell \in \mathbb{N}_0$,
	\begin{align*} P(S^t \geq \ell + 1)= \sum_{i=\ell}^\infty \pi_{i+1} 
		&= \theta^{-1} \sum_{i=\ell}^\infty P(S^i_+=i, S^i_-=0) \\
		&= \theta^{-1} P(S^i_+ \geq \ell, S^i_-=0) = \theta^{-1} P(S^i_{\pm}=\ell),
		\end{align*}
	by \eqref{eq:Npm:eqge}, which leads to \eqref{Eq:SifromSt}.
	\end{proof}

We are now able to relate the extremal index to the expected values of $S^i$ and $S^t$. 

\begin{theorem}	\label{prop:clustersize} Let $(I_t)_{t \in \mathbb{Z}}$ be as in Assumption \ref{Ass:1}.
	\begin{enumerate}[(a)]
		\item The three equalities $\theta = 0$,  $P(S^i_{\pm} =  \infty)=1$ and $P(S^i=\infty) = 1$ are equivalent. 	
	\item The inequalities $\theta>0$ and $P(S^i < \infty)>0$ are equivalent. If $\theta>0$, then 
	\begin{equation}\label{Eq:ThetaGen} \theta = \frac{ P(S^i_{\pm} < \infty) }{E(S^t)}. \end{equation} 
	If in addition also $P(S^i_{\pm} < \infty)=1$, then 
	\begin{align}
	\label{eq:clustersize:mean}
	\theta^{-1} &= E(S^t) \leq E(S^i)=1+2E(S^i_{\pm}), \\
	\label{eq:clustersize:second}
	E(S^i) &= \frac{E((S^t)^2)}{E(S^t)}.
	\end{align}
	The right-hand side in \eqref{eq:clustersize:mean} and both sides in \eqref{eq:clustersize:second} may equal infinity. Equality in \eqref{eq:clustersize:mean} only arises if there exists $k \in \mathbb{N}$ such that $P(S^i=k)=1$.
	\end{enumerate}
\end{theorem}

\begin{proof}
	We start with showing the equivalences stated in (a) and the first part of (b). Note first that if $P(S^i=k)>0$ for some $k \in \mathbb{N}$, then 
	$$\theta=P(S^i_{\pm}=0) \geq P(S^i_{\pm}=0 \mid S^i=k) P(S^i=k)>0 $$
	by Theorem~\ref{Th:FinClusUnif}(b) and so
	\begin{equation}\label{Eq:EquivI} P(S^i<\infty)>0 \; \Rightarrow \theta>0.
	\end{equation}
    Furthermore, 
    \begin{equation}\label{Eq:EquivII} P(S^i=\infty)=1  \; \Leftrightarrow \; P(S^i_\pm =\infty)=1,
    \end{equation}
	due to Theorem~\ref{Th:FinClusUnif}(c) (for the direction $\Rightarrow$) and since $S^i \geq S^i_+$ a.s. (for the direction $\Leftarrow$). Finally, 
	\begin{equation}\label{Eq:EquivIII} P(S^i_\pm =\infty)=1 \; \Rightarrow \; \theta=0, 
		\end{equation}
	since $\theta=P(S^i_\pm=0)$. Relations \eqref{Eq:EquivI}, \eqref{Eq:EquivII} and \eqref{Eq:EquivIII} together yield the desired equivalences. 
	
	We continue with the remainder of statement (b). Assume $\theta>0$. Since $P(S^t<\infty)=1$, Proposition \ref{Prop:ClusterPMF} and Theorem~\ref{Th:FinClusUnif}(a) allow us to write
		\begin{align*}
			E(S^t)=\sum_{k= 1}^\infty k \pi_k 
			&= \theta^{-1} \sum_{k= 1}^\infty k \cdot [P(S^i_{\pm} = k-1) - P(S^i_{\pm}= k)] \\
			&= \theta^{-1} \sum_{k=0}^\infty P(S^i_{\pm} = k) = \theta^{-1}P(S^i_{\pm}<\infty),
			\end{align*}
		which gives \eqref{Eq:ThetaGen}. Using the additional assumption $P(S^i_{\pm}<\infty)=1$ gives the first equality in \eqref{eq:clustersize:mean} and the last equality follows simply from \eqref{Eq:Splusminus}. Next we show \eqref{eq:clustersize:second} and start with
		\begin{align*}
		E\bigl((S^t)^2\bigr) &= \sum_{k=1}^\infty k^2  \pi_k = \theta^{-1} \sum_{k=1}^\infty k^2  [P(S^i_{\pm} = k-1) - P(S^i_{\pm}= k)] \\
		&= \theta^{-1} \sum_{k=1}^\infty k  P(S^i = k) = E(S^t) E(S^i),	
	    \end{align*}
 where we used Theorem~\ref{Th:FinClusUnif}(a) in the penultimate equation and $\theta^{-1}=E(S^t)$ in the final one. This yields \eqref{eq:clustersize:second}, where both sides may equal infinity. 
 
 Now, the central inequality in \eqref{eq:clustersize:mean} follows from \eqref{eq:clustersize:second} and the Cauchy--Schwarz inequality, since they together imply that
 $$ 1+2 E(S^i_{\pm})=E(S^i)=\frac{E((S^t)^2)}{E(S^t)} \geq E(S^t). $$
	Equality only holds if $\mbox{Var}(S^t)=0$ and thus there exists $k \in \mathbb{N}$ with $P(S^t=k)=1$. Now, if we assume that $\ell \in \mathbb{N}, \ell \neq k$ exists with $P(S^i=l)>0$ then we have by Theorem~\ref{Th:FinClusUnif}(b) that
	$$ P(S^t=\ell)=\frac{P(S^i=\ell, S_-^i=0)}{P(S_-^i=0)}=\frac{P(S^i=\ell)}{P(S_-^i=0)}P(S^i_-= 0|S^i=\ell)=\frac{P(S^i=\ell)}{\ell P(S_-^i=0)}>0, $$
	yielding a contradiction. Furthermore, as $E(S^i)=E(S^t)=k<\infty$ and thus $P(S^i<\infty)=1$, we have $P(S^i=k)=1$. 
	\end{proof}
Finally, we note that the inequality 
$$ E(S^t)\leq E(S^i) $$ 
always holds under the assumptions on Proposition~\ref{Prop:ClusterPMF}, as it either follows from Theorem~\ref{prop:clustersize} if $P(S^i_{\pm}<\infty)=1$ or from $E(S^i)=\infty$ if $P(S^i_{\pm}<\infty)<1$. 

\section{Examples}
We conclude with two examples to illustrate our findings and a remark about their statistical applications.

\begin{example}
	\label{ex:mm}
	Assume that 
	 $$ X_t = \max(Z_t, \ldots, Z_{t-r+1}), \;\;\; t \in \mathbb{Z},$$
	 for some $r \in \mathbb{N}$ and i.i.d.\ $Z_t$, $t \in \mathbb{Z},$ with $P(Z_t\leq x)=\exp(-1/x)$ for $x > 0$, i.e.\ $(X_t)_{t \in \mathbb{Z}}$ is a stationary moving maxima process with unit-Fr\'{e}chet distributed innovations. It can be shown (see \cite[Theorem 3]{Meinguet} for details) that $(X_t)_t$ is regularly varying and has tail process \eqref{Eq:TailProc} of the form
	 $$ Y_t= (Y \mathds{1}_{\{t+I \in \{1, \ldots, r\}\}}), \;\;\; t \in \mathbb{Z}, $$
	with $P(Y>x)=1-x^{-1}$ for  $x \geq 1$ and $Y$ is independent of $I$, which is uniformly distributed on $\{1, \ldots, r\}$. Thus, the inspected cluster size $S^i$ equals $r$ almost surely and
	\begin{equation*}
	(S^i_-, S^i_+) \overset{d}{=} (r-I, I-1).
	\end{equation*}
	which gives $\theta=1/r$. According to the final sentence in Theorem \ref{prop:clustersize}(a), the inequality ~\eqref{eq:clustersize:mean} becomes an equality and all expressions in it equal $r$. The inspected cluster size distribution $(\pi_k)_{k \ge 1}$ is degenerate at $r$, i.e.\ $\pi_k = \mathds{1}_{\{k=r\}}$.
\end{example}

\begin{example}
	\label{ex:urn}
	Urn models can be used to model the dynamic risk situation of a system, see, e.g., \cite{CirilloHuesler}. We restrict ourselves in this example to a very simple process in which at each time $t \in \mathbb{Z}$ a ball is randomly drawn with replacement from an urn filled with $g_n$ green balls, $y_n$ yellow balls and $r_n$ red balls, $g_n, y_n, r_n \in \mathbb{N}$, where the index $n \in \mathbb{N}$ allows us to look at a sequence of urn models with a varying composition of balls. The system modelled by urn $n$ is at risk at time $t \in \mathbb{Z}$, indicated by $I_t^n=1$, if either a red ball is drawn at time $t$ or if a red ball has been drawn at some time $s<t$ and no green balls have been drawn in $\{s+1, \ldots, t\}$. If the system is not at risk at time $t$, we set $I_t^n=0$. This way, $(I_t^n)_{t \in \mathbb{Z}}$ are stationary binary processes.
	
	Let now $r_n=o(y_n+g_n)$ with $r_n>0$, so that the event of drawing a red ball can be seen as an extreme event as $n \to \infty$. Assume furthermore that $g_n/(y_n+g_n) \to \rho \in (0,1]$ and define $(I_t)_{t \in \mathbb{N}}$ as in Assumption~\ref{Ass:1}. Due to the i.i.d.\ draws in the underlying urn models it is easily seen that $S^i_+$ (and therefore also $S^i_-$) has a $\operatorname{Geo}_0(\rho)$-distribution, i.e.\ $P(S^i_+=k)=\rho(1-\rho)^{k}$ for $k \in \mathbb{N}_0$. Thus, from Proposition~\ref{Prop:BeforeAfter} we obtain the joint distribution 
	$$ P(S^i_+ \geq k, S^i_- \geq \ell)=P(S^i_- \geq k+\ell)=(1-\rho)^{k+\ell}=P(S^i_- \geq k)P(S^i_+ \geq \ell), \;\;\; k,\ell \in \mathbb{N},$$
	so that $S^i_-$ are $S^i_+$ are independent and identically distributed. This independence then gives 
	$$ \pi_k=P(S^i=k \mid S^i_-=0)=P(S^i_+=k-1)=\rho (1-\rho)^{k-1}, \;\;\; k \in \mathbb{N},$$
	so $S^t \overset{d}{=} S^i_{\pm}+1$.
	We thus have
	$$ \theta^{-1}=E(S^t)=\frac{1}{\rho} \leq \frac{2-\rho}{\rho}= 1+ 2 E(S^i_-)=E(S^i), $$
	with strict inequality if $\rho<1$.
\end{example}
\begin{remark} In the previous example, the forward process $(I_t)_{t \geq 1}$ can easily be shown to be independent of the backward process $(I_t)_{t \leq -1}$. More generally, certain Markov-type structures of the underlying processes $(I_t^n)_{t \in \mathbb{N}}$ may lead to this independence property of the forward and backward processes. In this case, it follows that $S^i_+$ and $S^i_-$ are independent and therefore, from Proposition~\ref{Prop:BeforeAfter}, that for all $k, \ell \in \mathbb{N}_0$,
	$$ P(S^i_{\pm} \geq k+\ell|S^i_{\pm} \geq k)=\frac{P(S^i_+ \geq k, S^i_- \geq \ell)}{P(S^i_{\pm} \geq k)}=\frac{P(S^i_{\pm} \geq k)P(S^i_{\pm} \geq \ell)}{P(S^i_{\pm} \geq k)}=P(S^i_{\pm} \geq \ell), $$
which implies that the distribution of $S^i_{\pm}$ must be geometric. This makes our findings relevant for statistical inference as here some structural property of the underlying process allows us to conclude a parametric model for the inspected cluster size. 
	\end{remark}

\begin{acks}
The authors are grateful to Holger Drees and Holger Rootz\'{e}n for bringing this topic to their attention and for helpful discussions. They are also thankful to Hrvoje Planinić for his valuable comments on the manuscript. 
\end{acks}


\end{document}